\documentclass[12pt]{amsart}
\usepackage{amscd,amsmath,amsthm,amssymb}
\usepackage{tikz}
\tikzstyle{punkt}=[circle, fill=black, minimum size=1mm,inner sep=0pt, draw]

\def\NZQ{\mathbb}               

\def\PP{{\NZQ P}}

\def\KK{{\NZQ K}}

\def\G{{\mathcal G}}

\def\opn#1#2{\def#1{\operatorname{#2}}} 
\opn\chara{char} \opn\length{\ell} \opn\pd{pd} \opn\rk{rk}
\opn\projdim{proj\,dim} \opn\injdim{inj\,dim} \opn\rank{rank}
\opn\depth{depth} \opn\grade{grade} \opn\height{height}
\opn\embdim{emb\,dim} \opn\codim{codim}

\opn\Tr{Tr} \opn\bigrank{big\,rank}
\opn\superheight{superheight}\opn\lcm{lcm}
\opn\trdeg{tr\,deg}
\opn\reg{reg} \opn\lreg{lreg} \opn\ini{in} \opn\lpd{lpd}
\opn\size{size} \opn\sdepth{sdepth}
\opn\link{link}\opn\fdepth{fdepth}\opn\lex{lex}
\opn\tr{tr}
\opn\type{type}
\opn\gap{gap}
\opn\arithdeg{arith-deg}
\opn\astab{astab}
\opn\dstab{dstab}
\opn\bigheight{bigheight}
\opn\div{div} \opn\Div{Div} \opn\cl{cl} \opn\Cl{Cl}
\opn\Spec{Spec} \opn\Supp{Supp} \opn\supp{supp} \opn\Sing{Sing}
\opn\Ass{Ass} \opn\Min{Min}\opn\Mon{Mon} \opn{\ara}{ara} \opn{\rad}{rad}
\opn\Ann{Ann} \opn\Rad{Rad} \opn\Soc{Soc}
\opn\Im{Im} \opn\Ker{Ker} \opn\Coker{Coker} \opn\Am{Am}
\opn\Hom{Hom} \opn\Tor{Tor} \opn\Ext{Ext} \opn\End{End}
\opn\Aut{Aut} \opn\id{id}

\opn\nat{nat}
\opn\pff{pf}
\opn\Pf{Pf} \opn\GL{GL} \opn\SL{SL} \opn\mod{mod} \opn\ord{ord}
\opn\Gin{Gin} \opn\Hilb{Hilb}\opn\sort{sort}
\opn\PF{PF}\opn\Ap{Ap}
\opn\mult{mult}
\opn\aff{aff}
\opn\relint{relint} \opn\st{st}
\opn\lk{lk} \opn\cn{cn} \opn\core{core} \opn\vol{vol}  \opn\inp{inp} \opn\nilpot{nilpot}
\opn\link{link} \opn\star{star}\opn\lex{lex}\opn\set{set}
\opn\width{wd}
\opn\Fr{F}
\opn\QF{QF}
\opn\G{G}
\opn\type{type}\opn\res{res}
\opn\conv{conv}
\opn\gr{gr}

\def\pot#1#2{#1[\kern-0.28ex[#2]\kern-0.28ex]}

\opn\dirlim{\underrightarrow{\lim}}
\opn\inivlim{\underleftarrow{\lim}}
\def\Implies{\ifmmode\Longrightarrow \else
	\unskip${}\Longrightarrow{}$\ignorespaces\fi}
\def\implies{\ifmmode\Rightarrow \else
	\unskip${}\Rightarrow{}$\ignorespaces\fi}
\def\iff{\ifmmode\Longleftrightarrow \else
	\unskip${}\Longleftrightarrow{}$\ignorespaces\fi}

\let\:=\colon
\newtheorem{Theorem}{Theorem}[section]

\newtheorem{Corollary}[Theorem]{Corollary}
\newtheorem{Proposition}[Theorem]{Proposition}

\let\epsilon\varepsilon
\let\kappa=\varkappa
\def\qed{\ifhmode\textqed\fi
	\ifmmode\ifinner\quad\qedsymbol\else\dispqed\fi\fi}
\def\textqed{\unskip\nobreak\penalty50
	\hskip2em\hbox{}\nobreak\hfil\qedsymbol
	\parfillskip=0pt \finalhyphendemerits=0}
\def\dispqed{\rlap{\qquad\qedsymbol}}

\opn\dis{dis}
\def\pnt{{\raise0.5mm\hbox{\large\bf.}}}

\opn\Lex{Lex}

\begin{document}
	
	\title {Hankel edge ideals of trees and (semi-)Hamiltonian graphs}

	\author {Dariush Kiani, Sara Saeedi Madani, Saeed Tafazolian}

	\address{Dariush Kiani, Department of Mathematics and Computer Science, Amirkabir University of Technology (Tehran Polytechnic), Iran}
	\email{dkiani@aut.ac.ir}

	\address{Sara Saeedi Madani, Department of Mathematics and Computer Science, Amirkabir University of Technology (Tehran Polytechnic), Iran, and School of Mathematics, Institute for Research in Fundamental Sciences (IPM), Tehran, Iran}
	\email{sarasaeedi@aut.ac.ir}

	\address{Saeedi Tafazolian, Institute of Mathematics, Statistics and Computer Science, University of Campinas, 
	Rua Sergio Buarque de Holanda, 651, Cidade Universitria. SP, Brazil}
	\email{tafazolian@ime.unicamp.br}

	\dedicatory{ }
	
	\begin{abstract}
		In this paper, we study the Hankel edge ideals of graphs. We determine the minimal prime ideals of the Hankel edge ideal of labeled Hamiltonian and semi-Hamiltonian graphs, and we investigate radicality, being a complete intersection, almost complete intersection and set theoretic complete intersection for such graphs. We also consider the Hankel edge ideal of trees with a natural labeling, called rooted labeling. We characterize such trees whose Hankel edge ideal is a complete intersection, and moreover, we determine those whose initial ideal with respect to the reverse lexicographic order satisfies this property.      
	\end{abstract}
	
	\thanks{2020~{\em{Mathematics Subject Classification}}. 05E40.}

	
	
	
	
	\maketitle
	
	\setcounter{tocdepth}{1}
	
	\section{Introduction}\label{introduction}
	
 Let $\KK$ be a field and let $G$ be a finite simple graph (i.e. with no loops and multiple edges) with the vertex set $V(G)=[n]$ (i.e. $\{1,\ldots,n\}$) and the edge set $E(G)$. Then the binomial edge ideal $J_G$ of $G$ in the polynomial ring $R=\KK[x_1,\ldots,x_n,y_1,\ldots,y_n]$ is generated by the binomials $f_{ij}=x_iy_j-x_jy_i$ with $i<j$ such that $\{i,j\}\in E(G)$. This type of ideals were introduced at about the same time by Herzog et al. in \cite{HHHKR} and Ohtani in \cite{O} in~2010. The binomial edge ideal of $G$ could be seen as the ideal generated by a collection of $2$-minors of the matrix   
  \[
 \begin{bmatrix} x_1 & x_2 & \cdots & x_n \\ y_1 & y_2 & \cdots & y_n \end{bmatrix}.
 \]
 In the last decade, several algebraic and homological properties and invariants of binomial edge ideals have been intensively studied, mainly in terms of the combinatorial properties of the underlying graph, by many authors, see for example \cite{CR, EHH, EZ, HHO, KS, KS1, KS2, R, RSK1, S, SK}. 
 
 In 2015, another class of binomial ideals associated to graphs was introduced in \cite{CDE}. Let $S=\KK[x_1,\ldots,x_{n+1}]$ be the polynomial ring over $\KK$ and with the indeterminates $x_1,\ldots,x_{n+1}$. The \emph{Hankel edge ideal} of $G$, denoted by $I_G$, is the ideal of $S$ generated by the binomials $g_{ij}=x_ix_{j+1}-x_jx_{i+1}$ where $i<j$ and $\{i,j\} \in E(G)$. This ideal is also seen as the ideal generated by a collection of 2-minors of the $2\times n$ Hankel matrix
 \[
 X=\begin{bmatrix} x_1 & x_2 & \cdots & x_n \\ x_2 & x_3 & \cdots & x_{n+1} \end{bmatrix}.
 \]
 In the special case that $G$ is the complete graph with $n$ vertices, the Hankel edge ideal coincides exactly with the well known ideal $I_X$ of the rational normal curve $\mathcal{X} \subset \PP^n$. For some properties of the ideal $I_X$, see for example~\cite{BV, C, E}. 
 
 Note that in \cite{CDE} and \cite{DHI}, the ideal $I_G$ was called the \emph{binomial edge ideal of}~$X$ and also the \emph{scroll binomial edge ideal} of $G$. But, in this paper, we chose to call $I_G$ the Hankel edge ideal of $G$. 
 
 In \cite{CDE}, the authors determined all graphs $G$ for which the generators $g_{ij}$ form a Gr\"obner basis for $I_G$  with respect to the reverse lexicographic order $<$ induced by $x_1>\cdots >x_n>x_{n+1}$. Indeed, the only graphs with this property are the so-called \emph{closed} graphs (also known as proper interval graphs). A graph $G$ is said to be closed if one could label its vertices so that the maximal cliques (i.e. complete subgraphs) of $G$ are labeled as  intervals. Throughout this paper, when we talk about a closed graph, we mean, as usual, a closed graph with this specific labeling. 
 In \cite{CDE}, it was also shown that for any closed graph $G$, the Hankel edge ideal is Cohen-Macaulay of dimension~$c+1$ where $c$ is the number of connected components of $G$. In the same paper, the minimal prime ideals of $I_G$ were determined for any connected closed graph $G$. Consequently being radical as well as being a set theoretic complete intersection could be investigated for the same class of graphs. In \cite{CDE}, the authors also studied the Castelnuovo-Mumford regularity of $I_G$. They, indeed, gave a combinatorial upper bound for $\reg (S/I_G)$ where $G$ is a closed graph. Later, in \cite{DHI}, the authors characterized all closed graphs for which the aforementioned upper bound is attained. In the same paper, the graded Betti numbers of $I_G$ and $\ini_<(I_G)$ were also considered. In \cite{DHI}, a combinatorial characterization for Gorenstein Hankel edge ideals was also given. A generalization of Hankel edge ideals was also introduced and studied in \cite{CQ}. 
 
 As mentioned above, the class of closed graphs has played an essential role in the study of Hankel edge ideals of graphs so far. This is in fact because closed graphs admit a nice distinguished vertex labeling, and in general Hankel edge ideals do depend on the labeling of the underlying graphs. As an example for this fact, see \cite[page~972]{CDE}. In this paper, we study the Hankel edge ideals of some classes of graphs with certain natural labelings, like labeled Hamiltonian and semi-Hamiltonian graphs as well as rooted labeled trees, which are all defined precisely in the sequel. 
 
 The organization of this paper is as follows. In Section~\ref{hamiltonian}, we focus on the Hankel edge ideal of labeled Hamiltonian and semi-Hamiltonian graphs. By a labeled Hamiltonian graph, we mean a graph which has a Hamiltonian cycle (i.e. a cycle containing all the vertices of the graph) with a certain natural labeling of its vertices. By a labeled semi-Hamiltonian graph, we mean a non-Hamiltonian graph which admits a labeled Hamiltonian path (i.e. a path containing all the vertices of the graph) with a certain natural labeling of its vertices. Closed graphs are semi-Hamiltonian graphs, so that in this section, we recover most of the results in \cite{CDE}. However, this section deals with a much wider class of graphs than closed graphs. In this section, we determine the minimal prime ideals of the Hankel edge ideal of connected labeled (semi-)Hamiltonian graphs, and in particular, we show that the height of $I_G$ is $n-1$ where $|V(G)|=n$. Moreover, we characterize all Hankel edge ideals of connected labeled (semi-)Hamiltonian graphs, with respect to radicality, being complete intersection and almost complete intersection. In this section, we also show that the arithmetical rank of  $I_G$ is $n-1$, and $I_G$ is a set theoretic complete intersection for connected labeled (semi-)Hamiltonian graph $G$ with $|V(G)|=n$.  
 
  In the view of an observation in Section~\ref{hamiltonian} that if the Hankel edge ideal of a connected graph is a complete intersection, then the graph must be a tree, in Section~\ref{c.i.} trees become our main objects of interest. So, we need to fix a nice vertex labeling for trees which is provided in this section called \emph{rooted labeling}. Next, we show that under this labeling, the only trees which could have complete intersection Hankel edge ideal are paths. However, not all rooted labeled paths have this property. Indeed, we show that the Hankel edge ideal of a rooted labeled tree $T$ is a complete intersection if and only if $T$ is a path in which a leaf or a neighbor of a leaf is the root of $T$. Furthermore, we consider the initial ideal of such trees over $n$ vertices with respect to the reverse lexicographic order induces by $x_1>\cdots >x_n>x_{n+1}$. Indeed we show that there is only one path on $n$ vertices with a rooted labeling whose Hankel edge ideal has a complete intersection initial ideal. Throughout the paper, we also pose some questions accordingly.

   \section{Hankel edge ideal of Hamiltonian and semi-Hamiltonian graphs}\label{hamiltonian}
   
   In this section we determine the minimal prime ideals of the Hankel edge ideal of (semi-)Hamiltonian graphs with a certain labeling. As some consequences, we also study some of the algebraic properties of such ideals, like radicality, being a complete intersection, almost complete intersection and set theoretic complete intersection. 
   
   Recall that a \emph{Hamiltonian cycle} in a graph $G$ is a cycle which contains all the vertices of $G$. The graph $G$ which has a Hamiltonian cycle is called a \emph{Hamiltonian graph}. A \emph{Hamiltonian path} in $G$ is a path which contains all the vertices of $G$. The graph $G$ which has a Hamiltonian path is called \emph{traceabel graph}. A traceable graph which is not Hamiltonian is called \emph{semi-Hamiltonian graph}. We denote by $L_n$ the path with the vertex set $[n]$ and with the edges $\{i,i+1\}$ for all $i=1,\ldots , n-1$. We also denote by $C_n$ the cycle with the vertex set $[n]$ and with the edges $\{1,n\}$ and $\{i,i+1\}$ for all $i=1,\ldots , n-1$. 
   
   Now, let $G$ be a labeled graph with the vertex set $[n]$. If $C_n$ is a subgraph of $G$, then we say that $G$ is a \emph{labeled Hamiltonian graph}. If $L_n$ is a subgraph of $G$ and $\{1,n\}\notin E(G)$, then we say that $G$ is a \emph{labeled semi-Hamiltonian graph}. 
   
   As some well known classes of graphs which are labeled Hamiltonian, one could mention the cycles $C_n$, the complete graphs $K_n$ and the complete bipartite graphs $K_{t,t}$ with $n=2t$ where the even labels and odd labels provide the bipartition of the graph. The graph depicted in Figure~\ref{Hamiltonian} is also an example of a labeled Hamiltonian graph on the vertex set $\{1,2,3,4,5\}$ which is not in the aforementioned classes. 
   
   \usetikzlibrary{arrows}
   \begin{figure}[h!]
   	\begin{tikzpicture}[scale = 1]
   \draw [line width=1.6pt] (-5,4)-- (-3,4);
   \draw [line width=1.6pt] (-3,4)-- (-3,2);
   \draw [line width=2pt] (-3,2)-- (-5,2);
   \draw [line width=1.6pt] (-5,4)-- (-5,2);
   \draw [line width=1.6pt] (-3,4)-- (-1,3);
   \draw [line width=1.6pt] (-1,3)-- (-3,2);
   \draw (-5.52,4.42) node[anchor=north west] {$1$};
   \draw (-1,3.51) node[anchor=north west] {$3$};
   \draw (-3.08,4.59) node[anchor=north west] {$2$};
   \draw (-3,2.05) node[anchor=north west] {$4$};
   \draw (-5.5,2.1) node[anchor=north west] {$5$};
   \begin{scriptsize}
   	\draw [fill=black] (-5,4) circle (2pt);
   	\draw [fill=black] (-3,4) circle (2pt);
   	\draw [fill=black] (-3,2) circle (2pt);
   	\draw [fill=black] (-5,2) circle (2pt);
   	\draw [fill=black] (-1,3) circle (2pt);
   \end{scriptsize}
\end{tikzpicture}
\caption{A labeled Hamiltonian graph}
\label{Hamiltonian}
\end{figure}
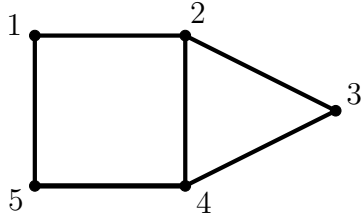
   
   As a well known class of graphs which are labeled Hamiltonian, one can mention non-complete closed graphs. Several properties of the Hankel edge ideal of this class of graphs were studied in \cite{CDE}. Recall that a graph $G$ is said to be closed if one could label its vertices so that the maximal cliques (i.e. complete subgraphs) of $G$ are labeled as  intervals. The graph shown in Figure~\ref{semiHamiltonian} is an example of a labeled semi-Hamiltonian graph with the vertex set $\{1,2,3,4,5,6\}$ which is not a closed graph. 
   
   \usetikzlibrary{arrows}
   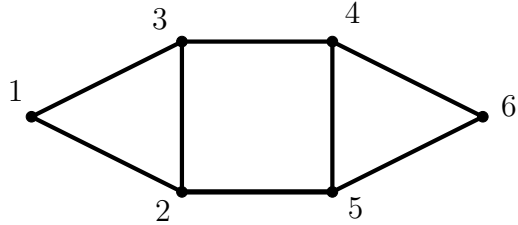
\begin{figure}[h!]
   	\begin{tikzpicture}[scale = 1]
   	\draw [line width=1.6pt] (-5,4)-- (-3,4);
   	\draw [line width=1.6pt] (-3,4)-- (-3,2);
   	\draw [line width=2pt] (-3,2)-- (-5,2);
   	\draw [line width=1.6pt] (-5,4)-- (-5,2);
   	\draw [line width=1.6pt] (-3,4)-- (-1,3);
   	\draw [line width=1.6pt] (-1,3)-- (-3,2);
   	\draw (-7.46,3.63) node[anchor=north west] {$1$};
   	\draw (-5.54,4.59) node[anchor=north west] {$3$};
   	\draw (-5.5,2.03) node[anchor=north west] {$2$};
   	\draw (-2.98,4.66) node[anchor=north west] {$4$};
   	\draw (-2.94,2.07) node[anchor=north west] {$5$};
   	\draw [line width=1.6pt] (-5,4)-- (-7,3);
   	\draw [line width=1.6pt] (-7,3)-- (-5,2);
   	\draw (-0.9,3.43) node[anchor=north west] {$6$};
   	\begin{scriptsize}
   	\draw [fill=black] (-5,4) circle (2pt);
   	\draw [fill=black] (-3,4) circle (2pt);
   	\draw [fill=black] (-3,2) circle (2pt);
   	\draw [fill=black] (-5,2) circle (2pt);
   	\draw [fill=black] (-1,3) circle (2pt);
   	\draw [fill=black] (-7,3) circle (2pt);
   	\end{scriptsize}
   	\end{tikzpicture}
   	\caption{A labeled semi-Hamiltonian graph}
   	\label{semiHamiltonian}
   \end{figure}

   \medskip
   In \cite{CDE}, it was shown that for any connected graph $G$ on $n$ vertices, $I_X$ is a minimal prime ideal of $I_G$. Therefore, since $\height I_X=n-1$ (see for example~\cite[Corollary~1.3]{CDE}), we have the following:
   
   \begin{Proposition}\label{height}
   	{\em (See} \cite[Proposition~2.1]{CDE}{\em )}
   	Let $G$ be a connected graph on $n$ vertices. Then the following hold:
   	\begin{itemize}
   		\item [{(a)}] $I_X$ is a minimal prime ideal of $I_G$.
   		\item [{(b)}] $\height I_G\leq n-1$.  
   		\item [{(c)}] For any minimal prime $P$ of $I_G$ which contains no variable, one has $P=I_X$. 
   	\end{itemize}
   \end{Proposition}
   
   Now, we determine the minimal prime ideals of $I_G$ whenever $G$ is a (semi-)Hamiltonian graph and we particularly deduce that for these graphs the height attains the exact value~$n-1$.  
   
   \begin{Theorem}\label{minimal primes}
 Let $G$ be a connected graph with $V(G)=[n]$. Then the following hold:
 \begin{itemize}
 	\item [{(a)}] If $G$ is a labeled Hamiltonian graph, then 
 	$\Min (I_G)=\{I_X\}$. 
 	\item [{(b)}] If $G$ is a labeled semi-Hamiltonian graph, then 
 	$\Min (I_G)=\{I_X,(x_2,\ldots, x_n)\}$. 
\end{itemize}	
In particular, in both cases, $\height I_G=n-1$. 
   \end{Theorem}   

\begin{proof}
Let $G$ be either a labeled Hamiltonian or a labeled semi-Hamiltonian graph, and let $P$ be a minimal prime ideal of $I_G$. First we observe that if $x_j\in P$, then $x_{j+1}\in P$, for any $j=1,\ldots,n-1$. Indeed, since $L_n$ is a subgraph of $G$, it follows that 
$g_{j j+1}=x_jx_{j+2}-x_{j+1}^2\in I_G$. Also, since $P$ is a prime ideal containing $I_G$, it follows from $x_j\in P$ that $x_{j+1}\in P$.	

(a)	By Theorem~\ref{height}, it is enough to show that $I_G$ does not have any minimal prime ideal containing variables. Suppose on contrary that $P$ is a minimal prime ideal of $I_G$ which contains a variable. Now let $i$ be the smallest integer with $x_i\in P$. First assume that $i=1$. The above observation implies that $x_j\in P$ for all $j=1,\ldots, n$. This implies that $I_X\subseteq (x_1,x_2,\ldots,x_n)\subseteq P$. Since $P$ and $I_X$ are both minimal prime ideals of $I_G$, it follows that $P=I_X$, which is a contradiction. Next assume that $i=2$. Then, again by our above observation, we have that $x_j\in P$ for all $j=2,\ldots,n$. On the other hand, since $C_n$ is a subgraph of $G$, we have $g_{1n}=x_1x_{n+1}-x_2x_n\in I_G$. Since $x_1\notin P$ and $P$ is a prime ideal containing $I_G$, we deduce that $x_{n+1}\in P$. Thus, we have $I_X\subseteq (x_2,x_3,\ldots,x_n,x_{n+1})\subseteq P$. Therefore, $P=I_X$ by minimality of $P$, which is a contradiction. Finally, assume that $i\geq 3$. Then $g_{i-2~i-1}=x_{i-2}x_i-x_{i-1}^2\in I_G\subseteq P$. Thus, $x_{i-1}\in P$ which is a contradiction, since $i$ is the smallest index such that $x_i\in P$. Therefore, $I_X$ is the only minimal prime ideal of $I_G$. 

(b) By Theorem~\ref{height}, it suffices to prove that if $P$ is a minimal prime ideal of $I_G$ containing a variable, then $P=(x_2,\ldots ,x_n)$. Let $i$ be the smallest integer with $x_i\in P$.  
If $i=1$, then by using our above observation it follows that $I_X\subseteq (x_1,\ldots,x_n)\subseteq P$, and hence $P=I_X$, a contradiction. If $i=2$, then again by our observation we have $(x_2,\ldots,x_n)\subseteq P$. Since $I_G\subseteq (x_2,\ldots,x_n)$ and $P$ is a minimal prime ideal of $I_G$, it follows that $P=(x_2,\ldots,x_n)$. If $i\geq 3$, then using the fact that $g_{i-2~i-1}\in I_G$, similar to the proof of part~(a), we deduce that $x_{i-1}\in P$, contradicting the assumption that $i$ is the smallest desired integer. Therefore, $\Min (I_G)$ consists only of $I_X$ and $(x_2,\ldots,x_n)$. 
\end{proof}       

For a graph $G$ and $e\in E(G)$, we denote by $G-e$ the subgraph of $G$ with the same vertex set as $G$ obtained by removing the edge $e$ from $G$. 

For a homogeneous ideal $I$ in $S$, we denote by $\mu(I)$, the cardinality of a minimal homogeneous generating set of $I$. 
We would like to remark that it is easily seen that neither of the binomials $g_{ij}=x_ix_{j+1}-x_{i+1}x_j$ for $1\leq i < j \leq n$ could be generated in $S$ by the other ones. In particular, the generators $g_{ij}$'s of $I_G$ provide a minimal generating set for this ideal. Therefore, we have $\mu(I_G)=|E(G)|$.    

\medskip
Using Theorem~\ref{minimal primes}, we obtain the following characterization of radicality of $I_G$ in the case of labeled (semi-)Hamiltonian graphs.

\begin{Corollary}\label{radical}
 Let $G$ be a connected graph with $V(G)=[n]$. Then the following hold:
 \begin{itemize}
 	\item [{(a)}] Let $G$ be a labeled Hamiltonian graph. Then $I_G$ is a radical ideal if and only if $G=K_n$. 
 	
 	\item [{(b)}] Let $G$ be a labeled semi-Hamiltonian graph. Then $I_G$ is a radical ideal if and only if $G=K_n-e$ where $e=\{1,n\}$. 
 \end{itemize}	
\end{Corollary}

\begin{proof}
(a) If $G=K_n$, then $I_G=I_X$ which is a prime ideal. Conversely, assume that $G$ is a labeled Hamiltonian graph and $I_G$ is radical, i.e. $I_G=\rad (I_G)$. Thus, by Theorem~\ref{minimal primes}~part~(a), we have $I_G=I_X$, and hence $\mu(I_G)=\mu(I_X)$ which implies that $|E(G)|=|E(K_n)|$. Thus $G=K_n$, as desired. 

(b) It was shown in \cite[Proposition~2.3]{CDE} that if $e=\{1,n\}$, then $I_{K_n-e}$ is a radical ideal with 
\begin{equation}\label{intersection}
I_{K_n-e}=I_X\cap (x_2,\ldots,x_n).
\end{equation} 

Conversely, assume that $G$ is a labeled semi-Hamiltonian graph and $I_G$ is radical. Thus, $I_G=\rad(I_G)=I_X\cap (x_2,\ldots,x_n)$, by Theorem~\ref{minimal primes}~part~(b). It follows from~\eqref{intersection} that $I_G=I_{K_n-e}$, and hence $G=K_n-e$.  
\end{proof}

Let $I$ be a homogeneous ideal in $S$. Recall that $S/I$ is a complete intersection if $\mu(I)=\height I$. 
Therefore, we have $S/I_G$ is a complete intersection if and only if $\height I_G=|E(G)|$. So, it follows from Theorem~\ref{minimal primes} that $S/I_G$ is never a complete intersection if $G$ is a labeled Hamiltonian graph, since $|E(G)|\geq n$ while $\height I_G = n-1$. 

Next, we determine those semi-Hamiltonian graphs $G$ for which $S/I_G$ is a complete intersection. First we have the following general observation:  

\begin{Proposition}\label{tree}
	Let $G$ be a connected graph such that $S/I_G$ is a complete intersection. Then $G$ is a tree and $\height I_G=n-1$.  
\end{Proposition}

\begin{proof}
	Let $|V(G)|=n$. Since $S/I_G$ is a complete intersection, we have $\height I_G=|E(G)|$. Proposition~\ref{height} implies that $|E(G)|\leq n-1$. Since $G$ is a connected graph with $n$ vertices, it follows that $|E(G)|=n-1$. This then implies that $G$ is a tree and $\height I_G=n-1$, as desired.   
\end{proof} 

\begin{Corollary}\label{complete intersection}
	Let $G$ be a connected labeled semi-Hamiltonian graph with $V(G)=[n]$. Then $S/I_G$ is a complete intersection if and only if $G=L_n$.  
\end{Corollary}

\begin{proof}
It is clear that $I_{L_n}$ is a complete intersection, since 
$\height I_G=|E(L_n)|=n-1$. Conversely, assume that $S/I_G$ is a complete intersection. Thus, by Proposition~\ref{tree}, $G$ is a tree.  
 But, the only tree with~$n$ vertices for which $L_n$ is a subgraph, is the path~$L_n$. Thus $G=L_n$, as desired.       	
\end{proof}

Let $I$ be a homogeneous ideal in $S$. Recall that $I$ is said to be an almost complete intersection if $\mu(I)=\height I + 1$. Therefore, we have $I_G$ is an almost complete intersection if and only if $\height I_G=|E(G)|-1$.

Also, recall that a graph is called \emph{unicyclic} if it contains exactly one cycle. A connected unicyclic graph is obtained from a tree by adding an edge between two non-adjacent vertices of the tree. 

\begin{Proposition}\label{almost complete intersection}
	Let $G$ be a connected graph with $V(G)=[n]$. Then the following hold:
	\begin{itemize}
		\item [{(a)}] Let $G$ be a labeled Hamiltonian graph. Then $I_G$ is an almost complete intersection if and only if $G=C_n$. 
		
		\item [{(b)}] Let $G$ be a labeled semi-Hamiltonian graph. Then $I_G$ is an almost complete intersection if and only if $G$ is a unicyclic graph obtained from $L_n$ by adding the edge $\{t,t+s\}$ for some $t,s$ with $1\leq t \leq n-2$ and $s\geq 2$. 
	\end{itemize}	
\end{Proposition}

\begin{proof}
By Theorem~\ref{minimal primes}, if $G$ is either a labeled Hamiltonian graph or a labeled semi-Hamiltonian graph, then the ideal $I_G$ is an almost complete intersection if and only if $|E(G)|=n$. The latter equivalently means that $G$ is a unicyclic graph, since $G$ is connected. 	
	
(a) Let $G$ be a labeled Hamiltonian graph. Then, the only unicyclic graph with~$n$ vertices which has $C_n$ as a subgraph is the cycle~$C_n$ itself. Hence, the statement follows. 

(b) Let $G$ be a labeled semi-Hamiltonian graph. Then, the only unicyclic graphs with $L_n$ as a subgraph are those ones obtained from $L_n$ by connecting two non-adjacent vertices of $L_n$ by an edge. This implies the desired result.      	
\end{proof}	

Since closed graphs are chordal, it follows immediately from Proposition~\ref{almost complete intersection} that:
	
\begin{Corollary}\label{closed}
Let $G$ be a closed graph on the vertex set $[n]$. Then $S/I_G$ is almost complete intersection if and only if $G$ is a unicyclic graph obtained from $L_n$ by adding the edge $\{t,t+2\}$ for some $t=1,\ldots, n-2$.  	
\end{Corollary}	

Let $I$ be an ideal of $S$. The \emph{arithmetical rank} of $I$, denoted by $\ara (I)$, is the least integer $r$ such that 
$\rad (I)=\rad (f_1,\ldots, f_r)$ for some $f_1,\ldots, f_r\in S$. It is well known that $\height I\leq \ara (I)$. The ideal $I$ is called a \emph{set-theoretic complete intersection} if $\height I = \ara (I)$.  

In \cite[Corollary~2.4]{CDE}, it was shown that the Hankel edge ideals of all connected closed graphs are set-theoretic complete intersection. In the next proposition, we generalize this result to all connected labeled Hamiltonian and semi-Hamiltonian graphs. 

\begin{Proposition}\label{set theoretic}
	Let $G$ be a connected labeled Hamiltonian or semi-Hamiltonian 
	graph with $V(G)=[n]$. Then $\ara (I_G)=n-1$. In particular, $I_G$ is a set-theoretic complete intersection. 
\end{Proposition}	

\begin{proof}
First note that $\ara (I_G)\geq n-1$, since $\height I_G=n-1$, by Theorem~\ref{minimal primes}. 

If $G$ is a labeled Hamiltonian graph, then it follows from~\cite[Proposition~1]{B} that $\ara (I_G)\leq n-1$, (see also \cite[Section~1]{RV} and \cite{V}). 

If $G$ is a labeled semi-Hamiltonian graph, then by Theorem~\ref{minimal primes} we have 
$\Min (I_G)=\Min (I_{L_n})=\{I_X,(x_2,\ldots,x_n)\}$, and hence 
$\rad (I_G)=\rad (I_{L_n})$. This implies that $\ara (I_G)\leq n-1$.  

Therefore, in both cases, we get the equality $\ara (I_G)=n-1$, as desired. The ``in particular" part is then immediate by definition.  	
\end{proof}

\usetikzlibrary{arrows}
\begin{figure}[h!]
	\begin{tikzpicture}[scale = 1]
	\draw [line width=1.6pt] (-3.66,2.99)-- (-2,3);
	\draw [line width=1.6pt] (-2,3)-- (0,3);
	\draw [line width=1.6pt] (0,3)-- (0,5.03);
	\draw [line width=1.6pt] (0,5.03)-- (-2,5);
	\draw [line width=1.6pt] (-2,5)-- (-2,3);
	\draw (-3.98,2.9) node[anchor=north west] {$1$};
	\draw (-2.18,2.9) node[anchor=north west] {$2$};
	\draw (-2.44,5.63) node[anchor=north west] {$3$};
	\draw (-0.05,5.63) node[anchor=north west] {$4$};
	\draw (-0.05,2.9) node[anchor=north west] {$5$};
	\begin{scriptsize}
	\draw [fill=black] (-3.66,2.99) circle (2pt);
	\draw [fill=black] (-2,3) circle (2pt);
	\draw [fill=black] (0,3) circle (2pt);
	\draw [fill=black] (0,5.03) circle (2pt);
	\draw [fill=black] (-2,5) circle (2pt);
	\end{scriptsize}
	\end{tikzpicture}
	\caption{A labeled semi-Hamiltonian graph with non-Cohen-Macaulay Hankel edge ideal}
	\label{CM}
\end{figure}
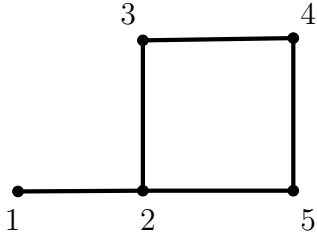	

We would like to remark that in \cite[Proposition~1.2]{CDE} it was shown that if $G$ is a closed graph, then $S/I_G$ is Cohen-Macaulay. However, this is not always the case in the more general case of labeled semi-Hamiltonian graphs. For example, let $G$ be the graph depicted in Figure~\ref{CM} which is a labeled semi-Hamiltonian non-closed graph. Our computations with \emph{Macaulay~2} show that $\projdim (S/I_G)=5$ while by Theorem~\ref{minimal primes} we have $\height I_G=4$. So, it would be interesting to ask which labeled semi-Hamiltonian graphs admit Cohen-Macaulay Hankel edge ideal.

	\section{Complete intersection Hankel edge ideals }\label{c.i.}
	
In the view of Proposition~\ref{tree}, now in this section we are interested in trees and in investigating about those trees $T$ for which $S/I_T$ is a complete intersection. 
	
Note that Proposition~\ref{tree}  is independent of the labeling of $G$. However, as we mentioned in Section~\ref{introduction}, in general the ideal $I_G$ strongly depends on the labeling of $G$. Therefore, now for trees we fix a natural labeling which we call \emph{rooted labeling}. First recall that $N_G(i)$ is the set of the neighbors of $i$, i.e. vertices adjacent to the vertex~$i$, and the degree of the vertex~$i$, denoted by $\deg (i)$ is equal to $|N_G(i)|$. We also set $N_G[i]=N_G(i)\cup \{i\}$.  

Let $T$ be a tree with $n$ vertices. Roughly speaking, to give a rooted labeling to $T$, we give the labels $1,\ldots,n$ to the vertices ``consecutively" as follows: we pick a vertex as the \emph{root} with the label~1, and then we label its neighbors in any order. Then, we label the neighbors of these new labeled vertices in the increasing order, and we continue this process to get all the vertices labeled. 

More precisely, pick a vertex as the root and give the label~1 to it. If $|N_T(1)|=t_1$, then give the labels $2,\ldots , t_1+1$ to the neighbors of~1 in an arbitrary order. Next, if $|N_T(2)|=t_2+1$, then label the $t_2$ unlabeled vertices in $N_T(2)$ by $t_1+2,\ldots, t_1+t_2+1$ in an arbitrary order. Similarly, for any $i$ with $3\leq i\leq t_1+1$, if $|N_T(i)|=t_i+1$, then label the $t_i$ unlabeled  vertices in $N_T(i)$ by $t_1+t_2+\cdots+t_{i-1}+2,\ldots , t_1+t_2+\cdots +t_i+1$ in an arbitrary order. Then repeat the same procedure for the neighbors of the vertices $i=t_1+2,\ldots, t_1+t_2+\cdots +t_{t_1+1}+1$. 
Then, by continuing this process, all the vertices of $T$ are labeled. 
Figure~\ref{Example} depicts an example of a tree with a rooted labeling. 

In a rooted labeled tree, we call those neighbors of a vertex $i$ whose  labels are greater than~$i$, the \emph{children} of $i$. For example, the children of the vertex $2$ in the graph of Figure~\ref{Example} are the vertices $5$ and $6$.  

\usetikzlibrary{arrows}
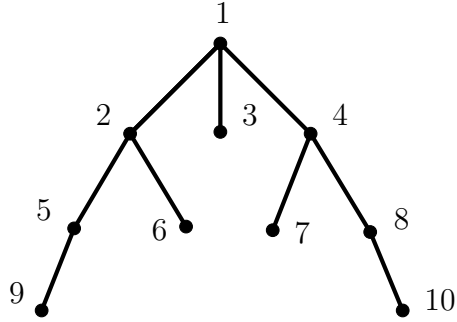
\begin{figure}[h!]
	\begin{tikzpicture}[scale = 1.2]
\draw [line width=1.6pt] (0.98,0.96)-- (-0.02,-0.04);
\draw [line width=1.6pt] (-0.02,-0.04)-- (0.98,0.96);
\draw [line width=1.6pt] (0.98,0.96)-- (0.98,-0.02);
\draw [line width=1.6pt] (0.98,0.96)-- (1.98,-0.04);
\draw (0.8,1.53) node[anchor=north west] {$1$};
\draw (-0.52,0.4) node[anchor=north west] {$2$};
\draw (1.1,0.4) node[anchor=north west] {$3$};
\draw (2.1,0.4) node[anchor=north west] {$4$};
\draw (1.68,-0.91) node[anchor=north west] {$7$};
\draw (2.78,-0.75) node[anchor=north west] {$8$};
\draw [line width=1.6pt] (1.98,-0.04)-- (2.64,-1.13);
\draw [line width=1.6pt] (1.98,-0.04)-- (1.56,-1.11);
\draw [line width=1.6pt] (-0.02,-0.04)-- (-0.64,-1.09);
\draw [line width=1.6pt] (-0.02,-0.04)-- (0.6,-1.07);
\draw [line width=1.6pt] (-0.64,-1.09)-- (-1,-2);
\draw [line width=1.6pt] (2.64,-1.13)-- (3,-2);
\draw (-1.48,-1.57) node[anchor=north west] {$9$};
\draw (-1.18,-0.65) node[anchor=north west] {$5$};
\draw (0.1,-0.87) node[anchor=north west] {$6$};
\draw (3.12,-1.65) node[anchor=north west] {$10$};
\begin{scriptsize}
	\draw [fill=black] (0.98,0.96) circle (2pt);
	\draw [fill=black] (-0.02,-0.04) circle (2pt);
	\draw [fill=black] (0.98,-0.02) circle (2pt);
	\draw [fill=black] (1.98,-0.04) circle (2pt);
	\draw [fill=black] (2.64,-1.13) circle (2pt);
	\draw [fill=black] (1.56,-1.11) circle (2pt);
	\draw [fill=black] (-0.64,-1.09) circle (2pt);
	\draw [fill=black] (0.6,-1.07) circle (2pt);
	\draw [fill=black] (-1,-2) circle (2pt);
	\draw [fill=black] (3,-2) circle (2pt);
\end{scriptsize}
\end{tikzpicture}
\caption{A tree with a rooted labeling}
\label{Example}
\end{figure}

\begin{Theorem} \label{rooted tree}
Let $T$ be a tree on $n$ vertices with a rooted labeling. If $T$ is not a path, then $\height I_T\leq n-2$. In particular, $S/I_T$ is not a complete intersection.      
\end{Theorem}

\begin{proof}
Since $T$ is not a path, it has a vertex of degree~$\geq 3$. First assume that the root~1 has degree at least~3. Thus, according to the labeling of $T$, we have that the vertices with the labels 2, 3 and 4 are the neighbors of~1. 

If $n\notin N_T[3]\cup N_T[4]$, then $P=(x_1,x_2,x_5,x_6,\ldots,x_n)$ is a prime ideal containing $I_T$ which implies that $\height I_T\leq n-2$. 

If $n\in N_T[3]\cup N_T[4]$, then we distinguish the following cases:

(1) If $n=4$, then $P=(x_1,x_2)$ is a prime ideal containing $I_T$, and hence $\height I_T\leq 2=n-2$. 

(2) If $n=5\in N_T[3]$, then $P=(x_1,x_2,g_{35})$ is a prime ideal containing $I_T$, since $(g_{35})$ is a prime ideal (of height~1) in $\KK[x_3,x_4,x_5,x_6]$. Thus, $\height I_T\leq 3=n-2$. 

(3) If $n=5\in N_T[4]$, then $P=(x_1,x_2,g_{45})$ is a prime ideal containing $I_T$, and hence $\height I_T\leq 3=n-2$.  

(4) If $n=6\in N_T[3]$, then $T$ is one of the trees $T_1, T_2, T_3$ in Figure~\ref{T1,T2,T3}. It is easily seen that in each case,  $P=(x_1,x_2,x_3,x_4)$ is a prime ideal containing $I_T$, and hence $\height I_T\leq 4=n-2$.  

\usetikzlibrary{arrows}
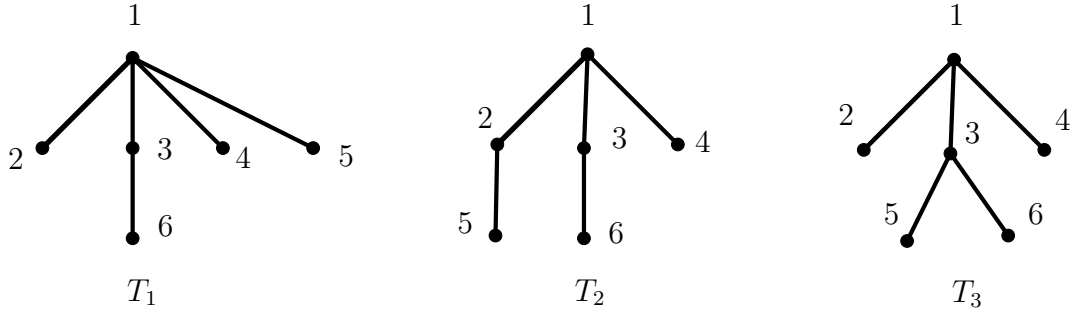
\begin{figure}[h!]
	\begin{tikzpicture}[scale = 1.2]
\draw [line width=1.6pt] (-6,2)-- (-7,1);
\draw [line width=2pt] (-7,1)-- (-6,2);
\draw [line width=1.6pt] (-6,2)-- (-6,1);
\draw [line width=1.6pt] (-6,2)-- (-5,1);
\draw [line width=1.6pt] (-6,1)-- (-6,0);
\draw [line width=1.6pt] (-6,2)-- (-4,1);
\draw [line width=1.6pt] (-0.96,2.04)-- (-1.96,1.04);
\draw [line width=2pt] (-1.96,1.04)-- (-0.96,2.04);
\draw [line width=1.6pt] (-0.96,2.04)-- (-1,1);
\draw [line width=1.6pt] (-0.96,2.04)-- (0.04,1.04);
\draw [line width=1.6pt] (-1.96,1.04)-- (-1.98,0.03);
\draw [line width=1.6pt] (-1,1)-- (-1,0);
\draw [line width=1.6pt] (3.1,1.98)-- (2.1,0.98);
\draw [line width=1.6pt] (2.1,0.98)-- (3.1,1.98);
\draw [line width=1.6pt] (3.1,1.98)-- (3.06,0.94);
\draw [line width=1.6pt] (3.1,1.98)-- (4.1,0.98);
\draw [line width=1.6pt] (3.06,0.94)-- (2.58,-0.03);
\draw [line width=1.6pt] (3.06,0.94)-- (3.7,0.03);
\draw (-6.18,2.71) node[anchor=north west] {$1$};
\draw (-7.5,1.11) node[anchor=north west] {$2$};
\draw (-5.85,1.23) node[anchor=north west] {$3$};
\draw (-4.98,1.13) node[anchor=north west] {$4$};
\draw (-3.84,1.15) node[anchor=north west] {$5$};
\draw (-5.84,0.37) node[anchor=north west] {$6$};
\draw (-1.16,2.71) node[anchor=north west] {$1$};
\draw (-2.3,1.57) node[anchor=north west] {$2$};
\draw (-0.81,1.35) node[anchor=north west] {$3$};
\draw (0.11,1.3) node[anchor=north west] {$4$};
\draw (-2.52,0.4) node[anchor=north west] {$5$};
\draw (-0.85,0.29) node[anchor=north west] {$6$};
\draw (2.92,2.71) node[anchor=north west] {$1$};
\draw (1.7,1.6) node[anchor=north west] {$2$};
\draw (3.1,1.4) node[anchor=north west] {$3$};
\draw (4.1,1.55) node[anchor=north west] {$4$};
\draw (2.2,0.5) node[anchor=north west] {$5$};
\draw (3.8,0.5) node[anchor=north west] {$6$};
\draw (-6.18,-0.33) node[anchor=north west] {$T_1$};
\draw (-1.22,-0.33) node[anchor=north west] {$T_2$};
\draw (2.96,-0.35) node[anchor=north west] {$T_3$};
\begin{scriptsize}
\draw [fill=black] (-6,2) circle (2pt);
\draw [fill=black] (-7,1) circle (2pt);
\draw [fill=black] (-6,1) circle (2pt);
\draw [fill=black] (-5,1) circle (2pt);
\draw [fill=black] (-6,0) circle (2pt);
\draw [fill=black] (-4,1) circle (2pt);
\draw [fill=black] (-0.96,2.04) circle (2pt);
\draw [fill=black] (-1.96,1.04) circle (2pt);
\draw [fill=black] (-1,1) circle (2pt);
\draw [fill=black] (0.04,1.04) circle (2pt);
\draw [fill=black] (-1.98,0.03) circle (2pt);
\draw [fill=black] (-1,0) circle (2pt);
\draw [fill=black] (3.1,1.98) circle (2pt);
\draw [fill=black] (2.1,0.98) circle (2pt);
\draw [fill=black] (3.06,0.94) circle (2pt);
\draw [fill=black] (4.1,0.98) circle (2pt);
\draw [fill=black] (2.58,-0.03) circle (2pt);
\draw [fill=black] (3.7,0.03) circle (2pt);
\end{scriptsize}
	\end{tikzpicture}
	\caption{The graphs $T_1$, $T_2$ and $T_3$}
	\label{T1,T2,T3}
\end{figure}

(5) If $n=6\in N_T[4]$, then $T$ is one of the trees $T'_1, T'_2, T'_3, T'_4$ in Figure~\ref{T'1,T'2,T'3,T'4}. It is easily seen that if $T=T'_1, T'_2, T'_4$, then $P=(x_1,x_2,x_4,x_5)$ is a prime ideal containing $I_T$, and if $T=T'_3$, then $P'=(x_1,x_2,x_4,x_6)$ is a prime ideal containing $I_T$. Therefore, we have $\height I_T\leq 4=n-2$.

\usetikzlibrary{arrows}
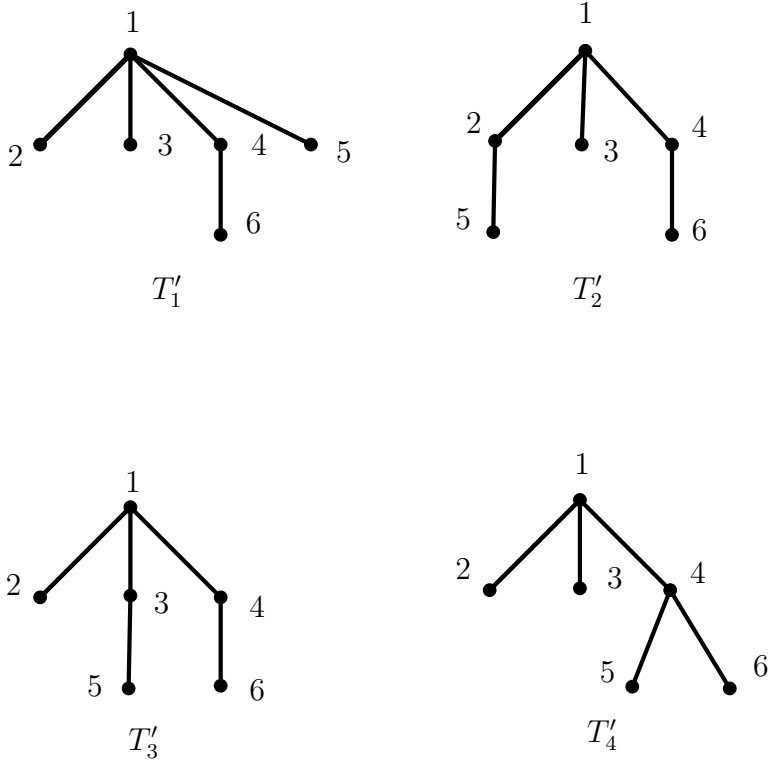
\begin{figure}[h!]
	\begin{tikzpicture}[scale = 1.2]
\draw [line width=1.6pt] (-1,2)-- (-2,1);
\draw [line width=2pt] (-2,1)-- (-1,2);
\draw [line width=1.6pt] (-1,2)-- (-1,1);
\draw [line width=1.6pt] (-1,2)-- (0,1);
\draw [line width=1.6pt] (-1,2)-- (1,1);
\draw [line width=1.6pt] (4.04,2.04)-- (3.04,1.04);
\draw [line width=2pt] (3.04,1.04)-- (4.04,2.04);
\draw [line width=1.6pt] (4.04,2.04)-- (4,1);
\draw [line width=1.6pt] (4.04,2.04)-- (5,1);
\draw [line width=1.6pt] (3.04,1.04)-- (3.02,0.03);
\draw [line width=1.6pt] (-1,-3.02)-- (-2,-4.02);
\draw [line width=1.6pt] (-2,-4.02)-- (-1,-3.02);
\draw [line width=1.6pt] (-1,-3.02)-- (-1,-4);
\draw [line width=1.6pt] (-1,-3.02)-- (0,-4.02);
\draw [line width=1.6pt] (-1,-4)-- (-1.02,-5.03);
\draw (-1.18,2.6) node[anchor=north west] {$1$};
\draw (-2.48,1.11) node[anchor=north west] {$2$};
\draw (-0.82,1.23) node[anchor=north west] {$3$};
\draw (0.22,1.23) node[anchor=north west] {$4$};
\draw (1.16,1.15) node[anchor=north west] {$5$};
\draw (0.16,0.37) node[anchor=north west] {$6$};
\draw (3.84,2.69) node[anchor=north west] {$1$};
\draw (2.6,1.47) node[anchor=north west] {$2$};
\draw (4.12,1.17) node[anchor=north west] {$3$};
\draw (5.1,1.43) node[anchor=north west] {$4$};
\draw (2.48,0.41) node[anchor=north west] {$5$};
\draw (5.1,0.29) node[anchor=north west] {$6$};
\draw (-1.18,-2.5) node[anchor=north west] {$1$};
\draw (-2.5,-3.65) node[anchor=north west] {$2$};
\draw (-0.86,-3.85) node[anchor=north west] {$3$};
\draw (0.2,-3.89) node[anchor=north west] {$4$};
\draw (-1.6,-4.77) node[anchor=north west] {$5$};
\draw (0.2,-4.81) node[anchor=north west] {$6$};
\draw (-0.88,-0.33) node[anchor=north west] {$T'_1$};
\draw (3.78,-0.33) node[anchor=north west] {$T'_2$};
\draw (-1.14,-5.33) node[anchor=north west] {$T'_3$};
\draw [line width=1.6pt] (0,1)-- (0,0);
\draw [line width=1.6pt] (5,1)-- (5,0);
\draw [line width=1.6pt] (0,-4.02)-- (0,-5);
\draw [line width=1.6pt] (3.98,-2.94)-- (2.98,-3.94);
\draw [line width=1.6pt] (2.98,-3.94)-- (3.98,-2.94);
\draw [line width=1.6pt] (3.98,-2.94)-- (3.98,-3.92);
\draw [line width=1.6pt] (3.98,-2.94)-- (4.98,-3.94);
\draw (3.8,-2.3) node[anchor=north west] {$1$};
\draw (2.48,-3.47) node[anchor=north west] {$2$};
\draw (4.16,-3.57) node[anchor=north west] {$3$};
\draw (5.08,-3.51) node[anchor=north west] {$4$};
\draw (4.08,-4.61) node[anchor=north west] {$5$};
\draw (5.78,-4.55) node[anchor=north west] {$6$};
\draw (3.94,-5.25) node[anchor=north west] {$T'_4$};
\draw [line width=1.6pt] (4.98,-3.94)-- (5.64,-5.03);
\draw [line width=1.6pt] (4.98,-3.94)-- (4.56,-5.01);
\begin{scriptsize}
	\draw [fill=black] (-1,2) circle (2pt);
	\draw [fill=black] (-2,1) circle (2pt);
	\draw [fill=black] (-1,1) circle (2pt);
	\draw [fill=black] (0,1) circle (2pt);
	\draw [fill=black] (1,1) circle (2pt);
	\draw [fill=black] (4.04,2.04) circle (2pt);
	\draw [fill=black] (3.04,1.04) circle (2pt);
	\draw [fill=black] (4,1) circle (2pt);
	\draw [fill=black] (5,1) circle (2pt);
	\draw [fill=black] (3.02,0.03) circle (2pt);
	\draw [fill=black] (-1,-3.02) circle (2pt);
	\draw [fill=black] (-2,-4.02) circle (2pt);
	\draw [fill=black] (-1,-4) circle (2pt);
	\draw [fill=black] (0,-4.02) circle (2pt);
	\draw [fill=black] (-1.02,-5.03) circle (2pt);
	\draw [fill=black] (0,0) circle (2pt);
	\draw [fill=black] (5,0) circle (2pt);
	\draw [fill=black] (0,-5) circle (2pt);
	\draw [fill=black] (3.98,-2.94) circle (2pt);
	\draw [fill=black] (2.98,-3.94) circle (2pt);
	\draw [fill=black] (3.98,-3.92) circle (2pt);
	\draw [fill=black] (4.98,-3.94) circle (2pt);
	\draw [fill=black] (5.64,-5.03) circle (2pt);
	\draw [fill=black] (4.56,-5.01) circle (2pt);
\end{scriptsize}
\end{tikzpicture}
\caption{The graphs $T'_1$, $T'_2$, $T'_3$ and $T'_4$}
\label{T'1,T'2,T'3,T'4}
\end{figure}

(6) If $n\geq 7$, then any vertex $i$ with $i\geq 5$ is a leaf. Thus, it is seen that $P=(x_1,x_2,x_3,x_4,x_5)$ is a prime ideal containing $I_T$, and hence $\height I_T\leq 5\leq n-2$. 

\medskip
Next, suppose that the degree of 1 is at most~2. Let $t\geq 2$ be the smallest label in $T$ with $\deg (t)\geq 3$, and let $t_1,\ldots,t_q$ be $q$ consecutive integers which are the labels of the children of $t$. 
First assume that there exists an $s=1,\ldots,q$ with $t_s > t+1$ such that $n\notin N_T(t_s)$. Let $s$ be the maximum integer with this property. 

If $t_s-1$ is not adjacent to 1, then we set 
$P=(x_2,x_3,\ldots, x_{t_s-1},x_{t_s+1},\ldots,x_n)$. Since in this case, $1$ is not adjacent to $t_s-1$ and $n$, and also $t_s$ is not adjacent to $n$, then it follows that $I_T\subseteq P$, and hence $\height I_T\leq n-2$. 

If $t_s-1$ is adjacent to 1, then it follows that $t_s=4$ and $t=2$. Therefore, 4 and 5 are the only neighbors of 2, and $n\geq 6$ is adjacent to 5. Thus, we set 
$P=(x_2,x_3,x_4,g_{5 6})$ in the case $n=6$, and $P=(x_2,\ldots, x_{n-1})$ in the case $n\geq 7$. Thus, it follows in both cases that $I_T\subseteq P$, and hence $\height I_T\leq n-2$.  

If such an $s$ does not exist, then it follows that $\deg (t)=3$ and the only children of $t$ are $t+1$ and $t+2$, and $n\in N_T(t+2)$. Now, we consider the following cases:
 
 (1) Suppose that $\deg (t+1)\geq 2$. Then, $t+3\notin N_T(t+2)$. We set $P=(x_2,x_3,\ldots, x_{t+2},x_{t+4},\ldots, x_n)$. Since $1$ is not adjacent to $t+2$ and since $n$ is a leaf with the neighbor $t+2$, it follows that $I_T\subseteq P$. 
 
 (2) Suppose that $\deg (t+1)=1$ and $n\geq t+4$. Then, we set $P=(x_2,x_3,\ldots, x_{n-1})$. Then, we have $I_T\subseteq P$.  
 
 (3) Suppose that $\deg (t+1)=1$ and $n=t+3$. Then, we set 
 $P=(x_2,x_3,\ldots, x_{t+1},g_{t+2,t+3})$. Then, we have $I_T\subseteq P$.
 
 In all the above cases, $P$ is a prime ideal in $S$ of height~$n-2$, which implies that $\height I_T\leq n-2$, and hence in neither of the cases $S/I_T$ is a complete intersection.           

The ``in particular" part is now immediate by Proposition~\ref{tree}. 	
\end{proof}

Now, by Theorem~\ref{rooted tree}, it remains to look at paths which have a rooted labeling. In the following we characterize all such paths with the complete intersection property. 

\begin{Theorem}\label{paths}
	Let $T$ be a path on $n$ vertices with a rooted labeling. Then
	 $S/I_T$ is a complete intersection if and only if the root of $T$ is either a leaf or the neighbor of a leaf. 
\end{Theorem}

\begin{proof}
 Suppose that $S/I_T$ is a complete intersection. Assume on contrary that the root of $T$ is neither a leaf nor the neighbor of a leaf. Thus, we have $n\geq 5$ and $N_T(1)=\{2,3\}$, $N_T(2)=\{1,4\}$ and $N_T(3)=\{1,5\}$. If $n\neq 6$, then it follows that the prime ideal $P=(x_2,x_3,x_5,x_6,\ldots, x_n)$ contains $I_T$. If $n=6$, then the prime ideal $P=(x_2,x_3,x_4,x_6)$ contains $I_T$. Therefore, in any case, $\height I_T\leq n-2$, and hence $S/I_T$ is not a complete intersection, a contradiction. 
 
 Conversely, first suppose that the root~1 is a leaf of $T$. Then, $T=L_n$, and hence $S/I_T$ is a complete intersection, by Corollary~\ref{complete intersection}.   
 
 Next, suppose that the root~1 is the neighbor of a leaf. Note that $N_T(1)=\{2,3\}$. Then, we distinguish two cases:

\medskip 
 Case~1. Assume that 2 is a leaf. If $n=3$, then it is clear that $\height I_T=2$, and the result is clear. Next, assume that $n\geq 4$. Let $P_1=(x_1,x_2,x_4,\ldots,x_n)$, $P_2=(x_2,x_3,\ldots,x_n)$ and $P_3=(x_1,x_2,g_{ij}: 3\leq i<j\leq n)$. It is clear that $I_T$ is contained in $P_i$ for all $i=1,2,3$. We claim that 
 \[
 \Min (I_T)=\{I_X,P_1,P_2,P_3\}.
 \]
 
 Let $Q\in \Min (I_T)$. If $Q$ contains no variables, then by Proposition~\ref{height}, $Q=I_X$. Suppose that $Q$ contains a variable. 
 
 First, assume that $x_2\notin Q$. Since $g_{12}\in I_T\subseteq Q$, it follows that $x_1, x_3\notin Q$. On the other hand, if $x_4\in Q$, then $x_2x_3\in Q$, because $g_{13}\in Q$, which is a contradiction. So, we assume that $i\geq 5$ is the smallest integer such that $x_i\in Q$. Now, it follows from $g_{i-2~i-1}=x_{i-2}x_i-x_{i-1}^2\in I_T\subseteq Q$ that $x_{i-1}\in Q$, a contradiction. 
 
Therefore, we assume that $x_2\in Q$. Then, $x_1x_3\in Q$, since $g_{12}=x_1x_3-x_2^2\in I_T\subseteq Q$. Thus, $x_1\in Q$ or $x_3\in Q$. Thus, we consider the following cases:
 
(i) Suppose that $x_1\in Q$. Let $T'$ be the induced subgraph of $T$ on the vertex set $[n]\setminus \{1,2\}$. Then, the Hankel edge ideal $I_{T'}$ is generated in $S'=\KK[x_3,\ldots,x_{n+1}]$. Thus, it follows from Theorem~\ref{minimal primes} (by a relabeling) that 
 \begin{equation}\label{min} 
 \Min_{S'} (I_{T'})=\{Q_1=(g_{ij}: 3\leq i<j\leq n),Q_2=(x_4,x_5,\ldots,x_n)\}.
 \end{equation}    
 
Since $Q\in \Min(I_T)$, we have $Q/(x_1,x_2)\in \Min (I_T+(x_1,x_2)/(x_1,x_2))$. This implies that 
\begin{equation}\label{Q'}
Q=Q'+(x_1,x_2),
\end{equation}
where $Q'$ is generated in $S'$. Since $I_{T'}+(x_1,x_2)\subseteq Q'+(x_1,x_2)$ and since $x_1$ and $x_2$ do not appear in the generators of $Q'$, it follows that $I_{T'}\subseteq Q'$ as ideals of $S'$, and moreover, we have $Q'\in \Min_{S'} (I_{T'})$. Therefore, by \eqref{min} we have $Q'=Q_1$ or $Q'=Q_2$, and hence \eqref{Q'} implies that $Q=P_1$ or $Q=P_3$.    
 
 (ii) Suppose that $x_3\in Q$ and $x_1\notin Q$. Since $g_{ii+1}\in I_T\subseteq Q$ for all $i=3,\ldots,n-1$, it follows that $x_i\in Q$ for all $i=3,\ldots,n$. Therefore, $Q=P_2$.
 
 Finally, since neither of $I_X$, $P_1$, $P_2$ and $P_3$ is contained in the others, the claim is proved. 

\medskip  
 Case~2. Assume that 2 is not a leaf and 3 is. Therefore, we have $n\geq 4$. Let $P_1=(x_1,x_2,x_4,\ldots,x_n)$, 
 $P_2=(x_2,x_3,\ldots,x_n)$, 
 $P_3=(x_1,x_2,x_3,g_{ij}:4\leq i<j\leq n)$ and $P_4=(x_1,x_2,x_3,x_5,\ldots,x_n)$. It is clear that $I_T$ is contained in $P_i$ for all $i=1,2,3,4$.
 We claim that if $n\neq 5$, then   
 \[
 \Min (I_T)=\{I_X,P_1,P_2,P_3,P_4\},
 \]
 and if $n=5$, then 
  \[
 \Min (I_T)=\{I_X,P_1,P_2,P_3\}.
 \]
Note that if $n=4$, then $P_3=P_4$. Let $Q\in \Min (I_T)$. 

First, suppose that $x_2\notin Q$. Since $g_{12}\in I_T\subseteq Q$, it follows that $x_1\notin Q$ and $x_3\notin Q$. We also have $x_4\notin Q$. Indeed, if $x_4\in Q$, then we have $x_2x_3\in Q$, since $g_{13}\in I_T\subseteq Q$, which is a contradiction. Now, let $i$ be the smallest integer such that $x_i\in Q$. Assume that $i=5$. Then, since $g_{24}\in Q$ and $x_3\notin Q$, it follows that $x_4\in Q$, a contradiction to the choice of $i$. Now, assume that $i\geq 6$. Then, since $g_{i-2~i-1}\in Q$, we deduce that $x_{i-1}\in Q$, a  contradiction to the choice of $i$. Therefore, if $x_2\notin Q$, then no other variable belongs to $Q$. Thus, by Proposition~\ref{height}, $Q=I_X$.     

Next, suppose that $x_2\in Q$. Since $g_{12}\in Q$, it follows that $x_1x_3\in Q$. Now, we consider the following cases: 

(i) Suppose that $x_1\in Q$ and $x_3\notin Q$. Since $x_2$ and $g_{24}$ belong to $Q$, we have $x_3x_4\in Q$, and hence $x_4\in Q$. Therefore, since $g_{ii+1}\in Q$ for any $i=4,\ldots,n-1$, we have $x_i\in Q$ for any $i=4,\ldots,n$. Hence, in this case, we deduce that $Q=P_1$. 

(ii) Suppose that $x_1\notin Q$ and $x_3\in Q$. Since $x_2$ and $g_{13}$ belong to $Q$, we have $x_1x_4\in Q$, and hence $x_4\in Q$. Therefore, similar to the previous case, we have $x_i\in Q$ for any $i=4,\ldots,n$. Hence, in this case, we deduce that $Q=P_2$.

(iii) Suppose that $x_1\in Q$ and $x_3\in Q$. Let $T'$ be the induced subgraph of $T$ on the vertex set $[n]\setminus \{1,2,3\}$. Then, the Hankel edge ideal $I_{T'}$ is generated in $S'=\KK[x_4,\ldots,x_{n+1}]$. Therefore, we have $Q=Q'+(x_1,x_2,x_3)$, where $Q'$ is a prime ideal generated in $S'$, and similar to Case~1, we have $Q'\in \Min_{S'}(I_{T'})$. If $n=4$, then $Q'=(0)$, and hence $Q=P_3=P_4$. If $n=5$, then $Q'=(g_{45})$, and hence $Q=P_3$. If $n\geq 6$, then by Theorem~\ref{minimal primes}, we have  
\[
Q'\in \{(x_5,\ldots,x_n), (g_{ij}:4\leq i<j\leq n)\},
\] 
and hence $Q=P_3$ or $Q=P_4$, which completes the proof. 
\end{proof}
	
Combining Theorem~\ref{rooted tree} and Theorem~\ref{paths}, we get the following characterization of rooted labeled trees with complete intersection Hankel edge ideal.

\begin{Corollary}\label{corollary}
	Let $T$ be a tree with a rooted labeling. Then
	$S/I_T$ is a complete intersection if and only if $T$ is a path whose root is either a leaf or the neighbor of a leaf. 
\end{Corollary}	
	
\medskip
Let $I$ be a monomial ideal in $S$. Recall that $S/I$ is a complete intersection if the monomial generators in the unique minimal monomial  generating set of $I$ are relatively prime.   

Let $<$ be the \emph{reverse lexicographic} order on $S$ induced by 
$x_1>\cdots >x_n >x_{n+1}$. By \cite[Theorem~1]{CDE}, we know that, not only $S/I_{L_n}$, but also $S/\ini_<(I_{L_n})$ is a complete intersection, since $\ini_<(I_{L_n})=(x_{i+1}^2: 1\leq i\leq n-1)$.

Let $T_1$ and $T_2$ be the rooted labeled paths on $[n]$, different from $L_n$, considered in Corollary~\ref{corollary}. Indeed, let $T_1$ be the rooted labeled path on $n\geq 3$ vertices with the root~1 in which the vertex~2 is a leaf, and let $T_2$ be the rooted labeled path on $n\geq 4$ vertices with the root~1 in which the vertex~3 is a leaf.  Now, according to Corollary~\ref{corollary}, it is natural to ask the same question about $S/\ini_<(I_{T_1})$ and $S/\ini_<(I_{T_2})$. First, in the following proposition, we determine the initial ideals of $I_{T_1}$ and $I_{T_2}$. 

\begin{Proposition} \label{initial}
	Let $T_1$ and $T_2$ be as above. Then, 
	\[
	\ini_<(I_{T_1})=(x_2x_3,x_1x_3^2,x_2^2,x_{i+1}^2: 3\leq i\leq n-1),
	\]
	and 
	\[
	\ini_<(I_{T_2})=(x_2x_3,x_3x_4,x_1x_3^2,x_1x_4^2,x_2^2,x_{i+1}^2: 4\leq i\leq n-1).
	\]
\end{Proposition} 	
	
\begin{proof}
 We prove the statement by applying Buchberger's criterion, see for example  \cite[Theroem~2.3.2]{HH}. Let $f=x_1x_2x_4-x_1x_3^2$. We first consider $T_1$.  We claim that the set of binomials 
 \[
 \mathcal{G}_1=\{f, g_{12}, g_{13}, g_{i i+1} : i=3,\ldots, n-1\} 
 \]
 is a Gr\"obner basis for $I_{T_1}$ with respect to $<$. 
 We have that $\ini_<(f)=x_1x_3^2$, $\ini_{<}(g_{12})=x_2^2$, $\ini_{<}(g_{13})=x_2x_3$ and $\ini_{<}(g_{i i+1})=x_{i+1}^2$ for $i=3,\ldots, n-1$. By \cite[Lemma~2.3.1]{HH}, we only need to compute $S(g_{12},g_{13})$ and $S(g_{13},f)$. It is easily seen that $S(g_{12},g_{13})=f$ and $S(g_{13},f)=-(x_1x_4) g_{12}$ which both reduce to zero with respect to $\mathcal{G}_1$. Therefore, the claim follows, and hence 
 $\ini_<(I_{T_1})=(x_2x_3,x_1x_3^2,x_2^2,x_{i+1}^2: 3\leq i\leq n-1)$, as desired.     
 
 Next, we consider $T_2$. Let $h=x_1x_3x_5-x_1x_4^2$. We claim that the set of binomials 
 \[
 \mathcal{G}_2=\{f, h, g_{12}, g_{13}, g_{24}, g_{i i+1} : i=4,\ldots, n-1\} 
 \]
 is a Gr\"obner basis for $I_{T_2}$ with respect to $<$. Note that  $\ini_<(h)=x_1x_4^2$, $\ini_<(g_{24})=x_3x_4$. Using \cite[Lemma~2.3.1]{HH} and according to our computations for $T_1$, we just need to compute the following four $S$-polynomials:
 \[
 S(f,h)=(x_3x_5)f+(x_2x_4)h, \quad \quad S(g_{13},g_{24})=-x_5 g_{12}+h,
 \]      
 
 \[
 S(g_{24},f)=-x_2 h, \quad \quad S(g_{24},h)=-x_5 f. 
 \]   	
 All the above $S$-polynomials clearly reduce to zero with respect to $\mathcal{G}_2$. Thus, the claim follows, and hence $\ini_<(I_{T_2})=(x_2x_3,x_3x_4,x_1x_3^2,x_1x_4^2,x_2^2,x_{i+1}^2: 4\leq i\leq n-1)$ which completes the proof. 
\end{proof}

The following corollary is now immediate. 
	
\begin{Corollary}\label{regular sequence}
	Let $T_1$ and $T_2$ be as above. Then $S/\ini_<(I_{T_1})$ and $S/\ini_<(I_{T_2})$ are not complete intersection.  
\end{Corollary}

We would like to close this section by posing some natural questions. 	In this section, we studied the trees $T$ with rooted labeling for which $S/J_T$ is a complete intersection. It is natural to ask if there is another vertex labeling for a tree $T$ for which $S/J_T$ is a complete intersection. On the other hand, a characterization of trees with rooted labeling whose Hankel edge ideal is almost complete intersection or set-theoretic complete intersection  would be of interest.

	\section*{Acknowledgment}
	
	The second author was in part supported by a grant from IPM (No. 1401130112). The third author was supported by CNPq-Brasil (Bolsa 310194/2019-9).

\end{document}